\documentclass[12pt,a4paper]{article}
\usepackage[utf8]{inputenc}
\usepackage[english]{babel}
\usepackage{latexsym,amssymb}
\usepackage{xcolor}
\usepackage{tensor}

\usepackage{cite}
\usepackage[pdftex]{graphicx}
\usepackage{tikz}
\usepackage{amsmath,amsthm,amsfonts,amssymb,bbm}
\usepackage{graphicx,psfrag,subfigure,float,color}
\usepackage{cite}
\usepackage{graphicx}
\usetikzlibrary{arrows}
\usepgflibrary{snakes}
\usetikzlibrary{arrows,snakes,backgrounds}
\DeclareGraphicsExtensions{.jpg,.pdf,.pdftex,.eps}

\newcommand{\Pb}{\mathbb{P}}

\newcommand{\dx}{\mathrm{d}}
\newcommand{\R}{\mathbb{R}}

\newcommand{\Z}{\mathbb{Z}}

\newcommand{\GUE}{\mathrm{GUE}}
\newcommand{\GOE}{\mathrm{GOE}}

\newtheorem{tthm}{Theorem}

\newtheorem{prop}{Proposition}[section]

\newtheorem{lem}[prop]{Lemma}
\newtheorem{defin}[prop]{Definition}
\newtheorem{cor}{Corollary}

\newtheorem{rem}[prop]{Remark}

\theoremstyle{definition}

\newcommand{\blocktheorem}[1]{%
  \csletcs{old#1}{#1}
  \csletcs{endold#1}{end#1}
  \RenewDocumentEnvironment{#1}{o}
    {\par\addvspace{1.5ex}
     \noindent\begin{minipage}{\textwidth}
     \IfNoValueTF{##1}
       {\csuse{old#1}}
       {\csuse{old#1}[##1]}}
    {\csuse{endold#1}
     \end{minipage}
     \par\addvspace{1.5ex}}
}

\usepackage{graphicx}
\usepackage{amssymb}
\usepackage{epstopdf}
\usepackage{nicefrac}
\author{P. Nejjar\thanks{Institute for Applied Mathematics, Bonn University, Endenicher Allee 60, 53115 Bonn, Germany. E-mail: {\tt nejjar@iam.uni-bonn.de}.
}}
\begin{document}

\title{KPZ statistics of second class particles   in ASEP via mixing}

\date{}

\maketitle 
\begin{abstract}
We consider the asymmetric simple exclusion process (ASEP) on $\Z$ with a single second class particle initially at the origin.
The first class particles form two rarefaction fans which come together at the origin,  where the large time density jumps from $0$ to $1$.
We are interested in $X(t)$, the position of the second class particle at time $t$.
We show that, under the KPZ $1/3$ scaling, $X(t)$ is asymptotically distributed as the difference of two independent, $\GUE$-distributed random variables. 
The key part of the proof is to show that $X(t)$ equals, up to a negligible term, the difference of a random number of holes  and  particles, with the randomness 
built up by ASEP itself.  This provides a KPZ analogue to the 1994  result of  Ferrari and Fontes \cite{FF94b}, where this randomness comes from the initial data and  leads to Gaussian limit laws.
\end{abstract}
\section{Introduction}

One of the most basic PDEs is the Burgers equation for $u(\xi,\tau) \in \R$ 
\begin{equation}\label{BE}
\partial_{\xi}u=-\theta\partial_{\tau}[u(1-u)], \quad \theta \in \R_{+}.
\end{equation}
As is standard in PDE theory, the equation \eqref{BE} may be solved by the method of characteristics. This may produce several solutions, but uniqueness comes if one additionally imposes an entropy condition \cite[Chapter 3]{Ev10}.

It is this entropy solution which is picked by the asymmetric simple exclusion process (ASEP): This is a continuous-time Markov process on $\{0,1\}^{\Z},$ we think of the $1'$s as particles and of the $0'$s as holes. The particles perform  independent nearest neighbor random walks, each particle jumps with probability $p>1/2$ a unit step to the right, and with probability $q=1-p<1/2$ a unit step to the left. However, particles can only jump to a site that is occupied by a hole. We can also think of the holes performing independent random walks, jumping to the right with probability $q$, to the left with probability $p$, and only being allowed to jump to sites occupied by a particle. This is the particle-hole duality. When $p=1$, we speak of the totally ASEP (TASEP).

Given a sequence of initial data $\zeta^{N}\in \{0,1\}^{\Z},N\geq 1,$  and $a<b,$ assume that 
\begin{equation*}
\lim_{N\to\infty}\frac{\#\{\mathrm{\,particles\,of\,}\zeta^{N}\mathrm{\, in \,}[aN,bN]\}}{N} =\int_{a}^{b}\dx \xi u_{0}(\xi).
\end{equation*}
Then, with $\zeta^{N}_{\tau N}$ being the state of the ASEP started from $\zeta^{N}$ at time $\tau N$, we have
\begin{equation}\label{hydro}
\lim_{N\to\infty}\frac{\#\{\mathrm{\,particles\,of\,}\zeta^{N}_{\tau N}\mathrm{\, in \,}[aN,bN]\}}{N}= \int_{a}^{b}\dx \xi u(\xi,\tau),
\end{equation}
where $u(\xi,\tau)$ is the entropy solution of \eqref{BE} with initial data $u_{0}$ and $\theta=p-q$.

Next to the convergence \eqref{hydro}, ASEP is also related to the Burgers equation by providing a microscopic analogue of the characteristics which carry the solution of the Burgers equation; this analogue is the second class particle.
This analogy is explained e.g. in \cite{F18}, \cite{Sep00}.
To define the second class particle, we can amend the state space of ASEP to equal $\{0,1,2\}^{\Z},$ where for $\zeta\in  \{0,1,2\}^{\Z},$ having  $\zeta(j)=2$ means that there is a second class particle at site $j$. We will only consider ASEPs with a single second class particle, and its dynamics are as follows: the second class particle interacts with holes like a particle, and with particles like hole. To distinguish the two types of particles, we say that if $\zeta(j)=1,$ $j$ is occupied by a first class particle.
Now when two characteristics meet, a shock (discontinuity) is created in the Burgers equation. If the Burgers equation has a shock wave that starts at the origin at time $0$ and propagates in time, then a second 
class particle initially placed at the origin will follow the shock wave. Being a microscopic object, it is of interest to study the fluctuations of the second class particle around the shock wave.

The main contribution of this paper is to characterize these fluctuations for a shock wave created by deterministic initial data. 
ASEP, when formulated as a growth model, belongs to the Kardar-Parisi-Zhang (KPZ) universality class, see \cite{Cor11} for an introduction to this class of models.
The KPZ class is conjectured to be governed by universal scaling exponents - $1/3$ for fluctuations, $2/3$ for correlations - and limit laws originating in random matrix theory.
There is a wealth of experimental evidence that this conjectural behavior soundly describes various growth phenomena,  we refer to the review article \cite{K17} as well as the research papers  \cite{TS12},\cite{TSSS11},  \cite{BDLYY13}.

As is shown by both experiments and mathematical results, the concrete limit law which arises in the KPZ class  depends on a few subclasses.
 Furthermore, if an additional source of randomness is present in a KPZ growth model (e.g. in the initial data),  it may well be that  this additional randomness
supersedes the randomness built up by the growth mechanism itself.  The absence of randomness in the initial data we consider  means that precisely this does not happen, i.e. by universality considerations  we expect the appearance of a $1/3$ fluctuation exponent and of random matrix limit laws.
Our main result, Theorem \ref{MAIN}, gives the first example of KPZ fluctuations of the second class particle  at shocks for ASEP with general asymmetry $p$.

Beyond the  convergence in distribution, we show that the rescaled position of the second class particle is asymptotically equal to the difference of a  random  number of holes and  particles, see Theorem  \ref{couplthm} in Section \ref{heur}.
We consider Theorem \ref{couplthm} to be the main conceptual novelty in this paper, and the bulk of the work in this paper is devoted to it.

For TASEP, several results for KPZ fluctuations of the second class particle at shocks have been obtained. Furthermore, in the general asymmetric case, but with a shock created by random initial data, the fluctuations of the second class particle (and much more) have been obtained: The second class particle then 
is determined by the random difference of the number of holes  and the number of particles present initially in a fixed segment.
 We can think of this paper as providing a KPZ counterpart to the Gaussian limit laws coming from the initial data.
We postpone the discussion of  previous results and how they relate to ours to Section \ref{previous}, and describe now our main result.

We will consider the ASEP $(\eta_{\ell})_{ \ell\geq 0}$ which has a single second class particle at the origin initially, i.e.
\begin{equation}
\eta_{0}(0)=2,
\end{equation}
and has first class particles starting from 
\begin{equation}\label{fc}
x_{n}(0)=
\begin{cases}
-n-\lfloor (p-q)(t- C t^{1/2})\rfloor    &\mathrm{if} \,  n \geq 1 \\
-n+1  &\mathrm{if}\, -\lfloor (p-q)(t  - C t^{1/2})\rfloor +1 \leq n \leq 0,
\end{cases}
\end{equation}
where $C\in \R$. To make the dependance of  $(\eta_{\ell})_{ \ell\geq 0}$ from $C,t$ clear we could write $(\eta_{\ell}^{t,C})_{ \ell\geq 0},$ but omit doing this to lighten the notation.
We will consider $(\eta_{\ell},0\leq \ell\leq t)$ and let $t$ go to infinity.
 In short, we have
\begin{equation}\label{IC}
\eta_{0}=2\mathbf{1}_{\{0\}}+\mathbf{1}_{\{x_{n}(0),n \geq- \lfloor (p-q)(t  - C t^{1/2})\rfloor +1\}}.
\end{equation}

The density profile of  \eqref{IC} thus has two macroscopic regions where the density is one. These two regions form a rarefaction fan, and at time $t$ these two fans come together for the first time. Thus a shock at the origin is created, where the density 
jumps from $0$ to $1$.  See  Figure \ref{one} for an illustration. We call this situation a hard shock, the fluctuations of a first class particle at the hard shock have been studied previously in \cite{N19}.

We denote by $X(\ell)$ the position the second class particle of $\eta_{\ell}$, and by $x_{n}(\ell)$ the position at time $\ell$ of the particle that started from $x_{n}(0)$.
In the following, we will not always write the integer parts.

We will consider the value  for the constant $C$
\begin{equation}\label{C}
C=C(M)=2\sqrt{\frac{M}{p-q}}, \quad M\in \Z_{\geq1}
\end{equation}
and  let $M$ go to infinity. Sending $M$ to infinity corresponds to soften the shock at the origin: As $M$ gets larger, more particles coming from the left arrive at the origin, and more particles starting close to the origin depart from it.
This softening is invisible on the hydrodynamic scale, but it is seen on the fluctuation level.

\begin{figure}
 \begin{center}
   \begin{tikzpicture}
       \draw (0.4,0.7) node[anchor=south]{\small{$1$}};
  \draw [very thick, ->] (0,-0.5) -- (0,2.5);
    \draw (-0.1,2) node[anchor=east] {\small{$u_{0}(\xi)$}};
    \draw[red,thick ] (-2,1.3) -- (-1,1.3);
      \draw[thick,red] (0,1.3) -- (1,1.3);
\draw[very thick] (-1,-0.1)--(-1,0.1)  ;
  \draw (-1,-0.6) node[anchor=south]{\small{$q-p$}};
  \draw[very thick] (1,-0.1)--(1,0.1)  ;
  \draw (1,-0.6) node[anchor=south]{\small{$p-q$}};
\filldraw(0,1.3) circle(0.08cm);
   \draw [very thick, ->] (-2,0) -- (2,0) node[below=4pt] {\small{$\xi$}};

\begin{scope}[xshift=7cm]
 \draw [very thick,->] (-2,0) -- (2,0) node[below=4pt] {\small{$\xi$}};
  \draw [very thick,->] (0,-0.5) -- (0,2.5);
    \draw (0.3,1.3) node[anchor=south]{\small{$1$}};
    \draw[red,thick] (-2,1.3) -- (0,0);
    \draw[red,thick] (0,1.3)--(2,0);
    \draw[very thick] (-1,-0.1) -- (-1,0.1);
       \draw (-0.8,-0.6) node[anchor=south]{\small{$q-p$}};
    \draw[very thick] (1,-0.1) -- (1,0.1);
       \draw (1.2,-0.6) node[anchor=south]{\small{$p-q$}};
\filldraw(0,1.3) circle(0.08cm); 
    \draw (-0.1,2) node[anchor=east] {\small{$u(\xi,1)$}};
   \end{scope}
   \end{tikzpicture}    \end{center} \caption{ Left: The macroscopic   initial particle density $u_{0}$ of    the initial configurations \eqref{IC}.
     Right: The large time particle density $u(\xi,\tau)$ at the macroscopic time  $\tau=1$. 
     This is the first macroscopic time where the two rarefaction fans come together, thus at the origin, $u$ jumps from $0$ to $1$, and $u(-\varepsilon,1),1-u(\varepsilon,1)>0$ for any $\varepsilon>0$. }\label{one}
\end{figure}
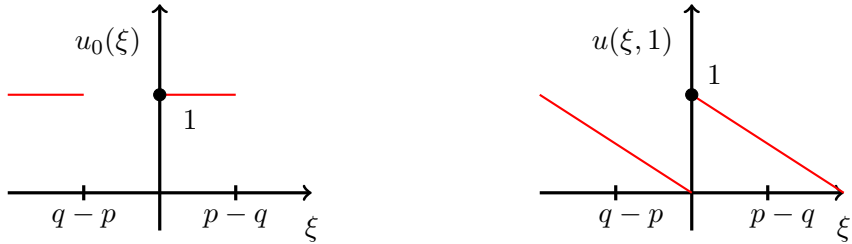

To state our  main result, we  define the Tracy-Widom $\GUE$ distribution, which first appeared in the theory of random matrices \cite{TW94}, as 
\begin{equation*}\label{FGUE2}
F_{\GUE}(s)=\sum_{n=0}^{\infty} \frac{(-1)^{n}}{n!}  \int_{s}^{\infty}\dx x_{1}\ldots  \int_{s}^{\infty}\dx x_{n}\det(K_{2}(x_{i},x_{j})_{1\leq i,j\leq n}),\end{equation*}
where $K_{2}(x,y)$ is the Airy kernel  $K_{2}(x,y)=\frac{Ai(x)Ai^{\prime}(y)-Ai(y)Ai^{\prime}(x)}{x-y},x\neq y, $ defined for $x=y$ by continuity and $Ai$ is the Airy function.  

The following Theorem gives the KPZ fluctuations of the second class particle at the hard shock in a double limit. 
\begin{tthm}\label{MAIN}
Let $C=C(M)$ be as in \eqref{C} and let  $X(t)$ be the position of the second class particle of $\eta_t$. Then for $s \in \R$ we have
\begin{equation*}
\lim_{M\to \infty}\lim_{t \to \infty}\Pb\left(\frac{X(t)}{M^{1/3}}\leq s\right)=\Pb(Y_{\GUE}^{\prime}-Y_{\GUE}\leq s),
\end{equation*}
where $Y_{\GUE}^{\prime},Y_{\GUE}$ are two independent, $\GUE-$distributed random variables.
\end{tthm}
The $M^{1/3}$ term is the typical KPZ $1/3$ fluctuation exponent.
To explain the $Y_{\GUE}^{\prime}-Y_{\GUE}$ limit law, note that inside each of the two rarefaction fans in Figure \ref{one}, the fluctuations of a particle follow a single $\GUE-$distribution (see Theorem \ref{TWT} and Proposition \ref{FGUE} afterwards).
The fluctuations to the left and right of the shock decouple asymptotically, and the second class particle is distributed as the difference of these two $\GUE$s.
We explain this picture more thoroughly in the following Sections \ref{previous} and \ref{heur}.
\begin{rem}
In Theorem \ref{MAIN}, we let $C$ go to plus infinity. It also meaningful to let $C$ go to minus infinity: Then, in the double limit $\lim_{C\to-\infty}\lim_{t\to\infty},$
the second class particle will behave like the second class particle started from $0$ with first class particles starting from $\mathbf{1}_{\Z_{\geq 1}},$ whose fluctuation behavior is given in Proposition \ref{yours}.
\end{rem}
\subsection{Comparison with other shocks}\label{previous}

Consider for $\rho \in [0,1), \lambda\in(\rho,1]$  an  initial data $\eta^{\rho,\lambda}$ where $(\eta^{\rho,\lambda}(i),i\in \Z)$ are independent random variables. 
We have $\Pb(\eta^{\rho,\lambda}(i)=1)=\rho=1-\Pb(\eta^{\rho,\lambda}(i)=0)$ for $i<0$ and 
 $\Pb(\eta^{\rho,\lambda}(i)=1)=\lambda=1-\Pb(\eta^{\rho,\lambda}(i)=0)$ for $i>0$.
There is a second class particle $X^{\rho,\lambda}$ starting from the origin.  In this case,
 there is initially  a shock at the origin which has speed $v=(p-q)(1-\lambda-\rho).$  
 A seminal result for $X^{\rho,\lambda}$ has been obtained in \cite{FF94b}.  Let, with $c_{\lambda,\rho}=(p-q)(\lambda-\rho),$
 $ \mathcal{P}^{\rho}$ be the random number of particles initially in $[-c_{\lambda,\rho}t,-1]$  and let $ \mathcal{H}^{\lambda}$ be the  number of holes initially in $[1,c_{\lambda,\rho}t],$ i.e.

 \begin{equation*} \mathcal{P}^{\rho}=\sum_{j=-c_{\lambda,\rho} t}^{-1}\eta^{\rho,\lambda}(j)
\quad \quad\mathcal{H}^{\lambda}=\sum_{j=1}^{c_{\lambda,\rho}t}1-\eta^{\rho,\lambda}(j).
 \end{equation*}
 It turns out that $ \mathcal{H}^{\lambda}, \mathcal{P}^{\rho}$ determine the behavior of $X^{\rho,\lambda}(t):$
As shown in  \cite[Theorem 1.1]{FF94b}, we have that 
 
 \begin{equation}\label{FF}
 \frac{(\lambda-\rho)X^{\rho,\lambda}(t)-\mathcal{H}^{\lambda}+\mathcal{P}^{\rho}}{t^{1/2}}\to 0
 \end{equation}
 in $L^{2}$. 
 
 Let $Y_{\mathcal{P}^{\rho}},Y_{\mathcal{H}^{\lambda}}$ be  independent, normally distributed random variables,   $Y_{\mathcal{P}^{\rho}}\sim \mathcal{N}(0,v_1),Y_{\mathcal{H}^{\lambda}}\sim \mathcal{N}(0,v_2),$ where  $v_1=(p-q)(\lambda-\rho)\rho(1-\rho)$, $ v_2=v_1\frac{\lambda(1-\lambda)}{\rho(1-\rho)}$.
 Using the standard central limit theorem, it is easy to deduce from \eqref{FF}
 that 
 \begin{equation}\label{gau}
 \frac{X^{\rho,\lambda}(t)-vt}{t^{1/2}}\Rightarrow Y_{\mathcal{H}^{\lambda}}-Y_{\mathcal{P}^{\rho}}
 \end{equation}
 where $\Rightarrow$ means convergence in distribution and of course $Y_{\mathcal{H}^{\lambda}}-Y_{\mathcal{P}^{\rho}}\sim \mathcal{N}(0,v_1+v_2).$ The convergence \eqref{gau} is a special case of  \cite[Theorem 1.2]{FF94b}.

While the shock we consider is different  than in the situation of \eqref{FF}, we can think of our results as analogues of \eqref{FF}, \eqref{gau}  in the absence of initial randomness. Instead of $\mathcal{P}^{\rho},\mathcal{H}^{\lambda}$
we have random variables $\mathcal{P}, \mathcal{H}$ (see \eqref{PH}) which also represent the random number of particles and holes which determine $X(t)$. However, in our case, the randomness of 
$\mathcal{P}, \mathcal{H}$  will be built up by ASEP itself as time evolves, and hence $\mathcal{P}, \mathcal{H}$   follow  the Tracy-Widom $\GUE$ distribution under the KPZ $1/3$  scaling.
The analogue of the statement  \eqref{FF} appears as Theorem \ref{couplthm} below, whereas we can think of Theorem \ref{MAIN} as a KPZ version of \eqref{gau}.
An  important conceptual difference to the shock in \eqref{FF} is that $ \mathcal{P}^{\rho},\mathcal{H}^{\lambda}$ are independent by definition, whereas  the independence of  $\mathcal{P},\mathcal{H}$ is true only asymptotically. Showing this asymptotic independence is non-trivial;
its statement appears as Theorem \ref{Theorem3}.

Concerning shocks   in the totally asymmetric case, the paper \cite{FGN18} considers a shock that on the hydrodynamic scale equals the one in \eqref{FF}, but is created by deterministic initial data.
Specifically, setting $x_{n}^{\rho,\lambda}=-\lfloor n/\rho\rfloor,n>0,$ and  $x_{n}^{\rho,\lambda}=-\lfloor n /\lambda\rfloor+1,n\leq0,$ and putting a second class particle $\tilde{X}^{\rho,\lambda}$ at the origin, we may rewrite 
\cite[Theorem 1.1]{FGN18} as follows\footnote{The roles of $\lambda,\rho$ are exchanged compared to \cite{FGN18}, and we performed an elementary computation to arrive at \eqref{simple}.}: 
\begin{equation}\label{simple}
2^{2/3}(\lambda-\rho)\frac{\tilde{X}^{\rho,\lambda}(t)-vt}{t^{1/3}}\Rightarrow \left(\rho(1-\rho)\right)^{2/3}Y_{\GOE}^{\prime}-\left(\lambda(1-\lambda)\right)^{2/3}Y_{\GOE},
\end{equation}
where $Y_{\GOE}^{\prime},Y_{\GOE}$ are two independent random variables distributed as the Tracy-Widom $\GOE$ distribution from random matrix theory. 

 There is a clear similarity between  \eqref{simple} and Theorem \ref{MAIN}. In particular, note that for the TASEP started from $x_{n}^{\rho,\lambda},n\in \Z$, the particle $x_{\kappa t}^{\rho,\lambda}(t)$ has $\GOE$ distributed fluctuations
 for all $\kappa$ except for $\kappa =\lambda\rho,$ in which case    $x_{\kappa t}^{\rho,\lambda}(t)$ is located at the shock. Thus, as in Theorem \ref{MAIN}, the limit law in \eqref{simple} is the difference of the asymptotically independent fluctuations to the left and the right of the shock. 
 However, the paper \cite{FGN18} crucially uses the coupling of the second class particle in TASEP to the competition interface in last passage percolation, which is only available in TASEP and does not provide a direct understanding  of $ \tilde{X}^{\rho,\lambda}$ as in Theorem \ref{couplthm} and \eqref{FF}.
 
 Finally, a similar picture arises in TASEP with a 'soft shock' between two rarefaction fans:
 By that we mean that at the orgin, two rarefaction fans come together, and the density makes a jump of size $at^{-1/3},a\in \R.$  The double limit $ \lim_{a\to+\infty}\lim_{t\to\infty}$
 thus corresponds to a hardening of the shock. The soft shock is invisible on the hydrodynamic scale,  and  \cite{Nej18} studied the fluctuations of a tagged  particle
 at the soft shock. Recently, \cite{BB19} also studied  the second class particle at the soft shock using a novel color-position symmetry.
It follows from \cite[Theorem 6.3]{BB19} that in the double limit $\lim_{a\to+\infty}\lim_{t\to\infty},$ the second class particle is distributed as the difference of two independent $\GUE$ distributions, in accordance with  our Theorem \ref{MAIN}.
 
  \subsection{Heuristics and method of proof}\label{heur}
Here we give a heuristics to understand why Theorem \ref{MAIN} is  true, by considering  the second class particle in a much simpler shock.

We have  the product blocking measure $\mu$ on $\{0,1\}^{\Z}$ with marginal
\begin{equation}
\mu(\{\zeta:\zeta(i)=1\})=\frac{(p/q)^{i}}{1+ (p/q)^{i}},
\end{equation}
which is concentrated on $\bigcup_{Z\in \Z}\Omega_{Z}$  with  
\begin{equation}\label{omega}
\Omega_{Z}=\{\zeta:\sum_{j<Z}\zeta(j)=\sum_{j\geq Z}1-\zeta(j)<\infty\}.
\end{equation}
An element of $\Omega_{Z}$ that will appear throughout the paper is what we call the reversed step initial data 
\begin{equation}\label{revstep}
\eta^{-\mathrm{step}(Z)}=\mathbf{1}_{\Z_{\geq Z}}.
\end{equation}

Define $\mu_{Z}=\mu(\cdot|\Omega_{Z})$ and we use the short hand notation $\mu_{Z}(i):=\mu_{Z}(\{\zeta:\zeta(i)=1\})$.
Then ASEP started from a $\zeta \in \Omega_{Z}$ (in particular, started from $\eta^{-\mathrm{step}(Z)}$),  is a countable state space  Markov chain, and it has $\mu_{Z}$ as stationary distribution \cite{Lig76}.
The $\mu_{Z}$ satisfy $\mu_{Z}(i)=\mu_{Z+k}(i+k)$ for arbitrary $Z,k$ \cite[special case of Corollary 6.3]{BB18}.
We define a $\Z$-valued random variable $V_{0}$ by 
\begin{equation}
\Pb(V_{0}=i)=\mu_{-1}(i)-\mu_{0}(i)=\mu_{0}(i+1)-\mu_{0}(i).
\end{equation}
With these information, one can deduce the following. 
\begin{prop}\label{yours}
Consider ASEP started from a $\zeta \in \Omega_{Z}$ and a second class particle  starting from a  $j \in \Z\setminus\{i:\zeta(i)=1\}$. 
Let $\tensor[^Z]{X}{_{}}(t)$ be the position of the second class particle at time $t$.
Then for $i \in \Z$  \begin{equation}\label{eq}
\lim_{t \to \infty}\Pb(\tensor[^Z]{X}{_{}}(t)=i)=\Pb(V_{0}+Z=i).
\end{equation}
\end{prop}
\begin{proof}Let     $\zeta^{1}(i)=\mathbf{1}_{\{i:\zeta(i)=1 \, \mathrm{or}\, i=\tensor[^Z]{X}{_{}}(0)\}}$ 
  so that $\zeta^{1}\in \Omega_{Z-1},\zeta\in \Omega_{Z}$ and we have
 \begin{equation}\label{TSW}\Pb(\tensor[^Z]{X}{_{}}(t)=i)=\Pb(\zeta^{1}_{t}(i)=1)-\Pb(\zeta_{t}(i)=1).\end{equation} 
 A proof of (the general fact)   \eqref{TSW} is e.g. provided in \cite{TW09e}.
  Use the convergence to $\mu_{Z-1},\mu_{Z}$ to conclude.
 \end{proof}
 From \eqref{eq} we obtain that for $t$ sufficiently large we have approximately 
\begin{equation}\label{approx}
\tensor[^Z]{X}{_{}}(t)    \approx Z+V_{0}.
\end{equation}

Let us try  to understand the second class particle in Theorem \ref{MAIN} 
by comparing it to the simple shock in Proposition \ref{yours}.
Compared to  ASEP started from $\eta^{-\mathrm{step}(1)}$ with  a second class particle at 0,  in the situation of Theorem  \ref{MAIN}, there are two additional mechanisms: There are the particles $x_{n},n\geq 1,$  (defined in \eqref{fc}) coming from the left which pull the second class particle to the left and there are the holes
$H_{n}(0),n \geq 1, $ (defined in \eqref{ICH}) coming from the right, which pull $X(t)$ to the right. Denote by $\mathcal{P}$ resp. $\mathcal{H}$ the random  number of particles resp. holes  which have interacted with the second class particle during $ [0,t]$. 
After particle $x_{\mathcal{P}}$ and hole $H_{\mathcal{H}}$  have interacted with $X(t)$ at some random time point $\tau\leq t$ no more new particles/holes arrive which could influence the position of $X(t)$. So after time $\tau$, we  might as well replace all particles to the left of  $x_{\mathcal{P}}$ by holes,  and all holes to the right of hole $H_{\mathcal{H}}$ by particles. The point is, by doing that, we get an ASEP particle configuration which lies in $\Omega_{\mathcal{H}-\mathcal{P}}$ and ask for the position of $X(t)$ in it.
Motivated by \eqref{approx}, we can then hope to have
\begin{equation}\label{heureq}
X(t) \approx \mathcal{H}-\mathcal{P}+V_{0}.
\end{equation}
Given \eqref{heureq} holds in a suitable sense, if we show that   $(M-\mathcal{H})M^{-1/3},(M-\mathcal{P})M^{-1/3}$ are asymptotically $\GUE$-distributed and independent,  and $V_{0}M^{-1/3}$ vanishes, this would 
prove Theorem \ref{MAIN}. However, there are of course several problems with \eqref{heureq}, specifically \eqref{approx} holds only for a fixed $Z$ and time going to infinity, 
here, $Z$ is randomly evolving during $[0,t]$ and it is unclear if  at time $t,$ the system has sufficiently relaxed to  (random) equilibrium to justify \eqref{heureq}.  Furthermore, it is unclear why $\mathcal{H},\mathcal{P}$ should be asymptotically independent.

Let us now give the rigorous definitions and theorems which justify the preceding heuristics. 
We label the holes of our initial data  \eqref{IC} as
\begin{equation}\label{ICH}
H_{n}(0)=
\begin{cases}
n+\lfloor (p-q)(t- C t^{1/2})\rfloor   \,\,  &\mathrm{for} \,  n \geq 1 \\
n-1\,\, &\mathrm{for}\, -\lfloor (p-q)(t  - C t^{1/2})\rfloor +1 \leq n \leq 0.
\end{cases}
\end{equation}

We define for $\chi\in(0,1/2),\chi^{\prime}\in (\chi,1/2)$ 
\begin{equation}\label{PH}
\begin{aligned}
&\mathcal{P}=\sup\{i\in \Z| x_{i}(t-t^{\chi})>-t^{\chi^{\prime}}\}
\\&\mathcal{H}=\sup\{i\in \Z| H_{i}(t-t^{\chi})<t^{\chi^{\prime}}\}.
\end{aligned}
\end{equation}
It is easy to see that almost surely, $\mathcal{H},\mathcal{P}$ are finite, so the $\sup$ is in fact a $\max$.
While $\mathcal{H},\mathcal{P}$ depend on $\chi,\chi^{\prime},$ their asymptotic behaviour does not. The roles of $\chi,\chi^{\prime}$ are explained in Section \ref{5}, see in particular Figure \ref{Graph} therein. Here is the rigorous version of \eqref{heureq}: 

\begin{tthm}\label{couplthm}
Consider $X(t)$ with $C=C(M)$ from \eqref{C} and $\mathcal{P},\mathcal{H}$ from \eqref{PH}. Then for $\varepsilon>0$ 
\begin{equation}
\lim_{M\to \infty}\lim_{t\to\infty}\Pb\left( |X(t)-\mathcal{H}+\mathcal{P}|>M^{\varepsilon}\right)=0.
\end{equation}
\end{tthm}
Important tools for proving Theorem \ref{couplthm} are couplings, bounds on inter particle distances, and bounds on the mixing times of countable state space ASEPs, see Section \ref{countable} for the latter. Theorem \ref{couplthm} is proven in Section \ref{5}.

The following theorem gives the limit law of $\mathcal{H}-\mathcal{P}$:
\begin{tthm}\label{Theorem3}
We have for $s \in \R$  and $C=C(M)$ as in \eqref{C}
\begin{equation*}
\lim_{M\to \infty}\lim_{t \to \infty}\Pb\left(\frac{\mathcal{H}-M}{M^{1/3}} -  \frac{\mathcal{P}-M}{M^{1/3}}\leq s\right)=\Pb(Y_{\GUE}^{\prime}-Y_{\GUE}\leq s),
\end{equation*}
where $Y_{\GUE}^{\prime},Y_{\GUE}$ are two independent, $\GUE-$distributed random variables.
\end{tthm}
A key tool used to prove Theorem \ref{Theorem3} are comparisons with ASEPs with step initial data as well as the slow-decorrelation method \cite{Fer08},\cite{CFP10b}. Theorem \ref{Theorem3} is proven in Section \ref{4}.

\subsection{Proof of Theorem \ref{MAIN}}
Given Theorem \ref{couplthm} and Theorem \ref{Theorem3}, it is not very hard to  prove our main result Theorem \ref{MAIN}:
\begin{proof}[Proof of Theorem \ref{MAIN}]
Define for $\varepsilon<1/3$ 
\begin{equation*}
K_{t,\varepsilon}=\{|X(t)-\mathcal{H}+\mathcal{P}|\leq M^{\varepsilon}\}
\end{equation*}
Then, using Theorem \ref{couplthm},
we have 
\begin{align*}
\lim_{M\to \infty}\lim_{t\to\infty}\Pb\left( \frac{X(t)}{M^{1/3}} \leq s \right)=\lim_{M\to \infty}\lim_{t\to\infty}\Pb\left( \bigg\{\frac{X(t)}{M^{1/3}}\leq s\bigg\} \cap K_{t,\varepsilon}\right).
\end{align*}
Consequently, we obtain the inequalities 
\begin{align*}
\lim_{M\to \infty}     \lim_{t\to\infty}    \Pb\left( \frac{X(t)}{M^{1/3}}  \leq s \right)\leq \lim_{M\to \infty}\lim_{t\to\infty}\Pb\left( \frac{\mathcal{H}-\mathcal{P}}{M^{1/3}}\leq s +M^{\varepsilon-1/3}\right)
\end{align*}
and 
\begin{align*}
\lim_{M\to \infty}\lim_{t\to\infty}\Pb\left( \frac{\mathcal{H}-\mathcal{P}}{M^{1/3}}\leq s -M^{\varepsilon-1/3}\right)
\leq \lim_{M\to \infty}\lim_{t\to\infty}\Pb\left( \frac{X(t)}{M^{1/3}}\leq s \right).
 \end{align*}
 This finishes the proof using Theorem \ref{Theorem3}.
\end{proof}
\subsection{Labelings and basic coupling}\label{lbc}
It will be important for us that we may construct a second class particle as the position where two ASEPs with first class particles differ.
We will define two ASEPs $(\eta_{\ell}^{1})_{\ell\geq 0},(\eta_{\ell}^{2})_{\ell\geq 0}$ on $\{0,1\}^{\Z}$ which differ from $\eta_{0}$ in that 
 $\eta_{0}^{1}$ has a first class particle instead of the second class particle,
 and $\eta_{0}^{2}$ has a hole instead.
 For later usage, let us explicitly write down the labelings of the particles of  $\eta_{0}^{1}$ as 
\begin{equation}\label{IC1}
x_{n}^{1}(0)=
\begin{cases}
-n-\lfloor (p-q)(t- C t^{1/2})\rfloor   \quad &\mathrm{for} \,  n \geq 1 \\
-n\quad &\mathrm{for}\, -\lfloor (p-q)(t  - C t^{1/2})\rfloor  \leq n \leq 0,
\end{cases}
\end{equation}
and the holes of $\eta_{0}^{2}$ as 
\begin{equation}\label{H2}
H_{n}^{2}(0)=
\begin{cases}
n+\lfloor (p-q)(t- C t^{1/2})\rfloor   \quad &\mathrm{for} \,  n \geq 1 \\
n\quad &\mathrm{for}\, -\lfloor (p-q)(t  - C t^{1/2})\rfloor  \leq n \leq 0.
\end{cases}
\end{equation}
So in short, we have
\begin{equation*} \eta_{0}^{1}=\mathbf{1}_{\{x_{n}^{1}(0),n\geq  -\lfloor (p-q)(t  - C t^{1/2})\rfloor\}}\quad  \eta_{0}^{2}=\mathbf{1}_{\Z}-\mathbf{1}_{\{H_{n}^{2}(0),n\geq  -\lfloor (p-q)(t  - C t^{1/2})\rfloor\}}.\end{equation*}
Initially, $ \eta_{0}^{1},\eta_{0}^{2}$ differ only at position zero.
Under the basic coupling (i.e. the dynamics of  $ \eta_{\ell}^{1},\eta_{\ell}^{2}$ are constructed using the same poisson processes),   $ \eta_{\ell}^{1},\eta_{\ell}^{2}$  differ exactly at one random position for all $\ell\geq 0$. We can define the position of the second class particle $X(t)$ of $\eta_{t}$ as the position of this discrepancy:
\begin{equation}\label{X(t)}
X(t):=\sum_{j\in \Z}j\mathbf{1}_{\{\eta_{t}^{1}(j)\neq \eta_{t}^{2}(j)\}}.
\end{equation}
An important simple observation is that we always have \[\eta_{t}^{1}(X(t))=1>\eta_{t}^{2}(X(t))=0.\]

\subsection{Outline}
In Section \ref{compare}, we introduce ASEP with step initial data and the main fluctuation result we need for it. ASEPs with shifted step initial data will play an important role to obtain the fluctuations of $\mathcal{P},\mathcal{H}$.
 In Section \ref{countable} we provide bounds on the position of holes/particles in countable state space ASEPs that we will use. In Section \ref{4}, Theorem \ref{Theorem3} is proven, and in Section \ref{5}, Theorem \ref{couplthm} is proven.

\section{ASEPs with step initial data}\label{compare}
By ASEP with step initial data  we mean the initial configuration $x_{n}^{\mathrm{step}}(0)=-n,n\geq 1.$ We use as input that the fluctuations in ASEP with step initial data have been obtained.
Let us start with the definition of the relevant distribution functions. 
\begin{defin}[ \cite{TW08b},\cite{GW90}] Let $s \in \R,M \in \Z_{\geq 1}$. We define for $p\in (1/2,1)$  
\begin{equation} \label{def1}
F_{M,p}(s)=\frac{1}{2 \pi i} \oint  \frac{\dx \lambda}{\lambda} \frac{\det(I-\lambda K)}{\prod_{k=0}^{M-1}(1-\lambda (q/p)^{k})}
\end{equation}
where $K= \hat{K}\mathbf{1}_{(-s,\infty)}$  and 
$\hat{K}(z,z^{\prime})=\frac{p}{\sqrt{2 \pi}}e^{-(p^{2}+q^{2})(z^{2}+z^{\prime 2})/4+pqzz^{\prime}}$ and the integral is taken over a counterclockwise oriented contour enclosing the
poles $\lambda=0,\lambda=(p/q)^{k},k=0,\ldots,M-1$.  For $p=1$, we define
\begin{equation*}
F_{M,1}(s)=\Pb\left(\sup_{0=t_{0}<\cdots <t_{M}=1}\sum_{i=0}^{M-1}[B_{i}(t_{i+1})-B_{i}(t_{i})]\leq s    \right),
\end{equation*}
where $B_{i},i=0,\ldots,M-1$ are independent standard Brownian motions. \end{defin}

What is important to us is that $F_{M,p}$ arises as limit law in ASEP with step initial data.
\begin{tthm}[Theorem 2 in \cite{TW08b},  Corollary 3.3 in \cite{GW90}]\label{TWT}
Consider $x_{n}^{\mathrm{step}}(0)=-n,n\geq 1$. We have for $M\in \Z_{\geq 1}$  that 
\begin{equation*}
\lim_{t \to \infty}\Pb(x_{M}^{\mathrm{step}}(t)\geq(p-q)(t-st^{1/2}))=F_{M,p}(s).
\end{equation*}
\end{tthm}
Finally, we need to know that $F_{M,p}$ converges to $F_{\GUE}$ on the right scale.

\begin{prop}[Proposition 2.1 in \cite{N19}]\label{FGUE}

Let $s \in \R.$  Then we have 
\begin{equation}\label{two1}
\lim_{M \to \infty}F_{M,p}\left(\frac{2\sqrt{M}+sM^{-1/6}}{\sqrt{p-q}}\right)=F_{\GUE}(s).
\end{equation}

\end{prop}
We will often compare the particles and holes  $x_{n},H_{n},n\geq 1$ with the ones defined in \eqref{A},\eqref{B} below.
Let $(\eta^{A}_{t})_{t\geq 0}$ be the ASEP started from $\eta^{A}_{0}=\mathbf{1}_{\{x_{n}^{A}(0),n\geq1\}}$ with 
\begin{equation}\label{A}
x_{n}^{A}(0)=
-n-\lfloor (p-q)(t- C t^{1/2})\rfloor   \quad \mathrm{for} \,  n \geq 1,
\end{equation}
and let  $(\eta^{B}_{t})_{t\geq 0}$ be the ASEP started from $\eta^{B}_{0}=\mathbf{1}_{\Z}-\mathbf{1}_{\{H_{n}^{B}(0),n\geq1\}}$ with 
\begin{equation}\label{B}
H_{n}^{B}(0)=
n+\lfloor (p-q)(t- C t^{1/2})\rfloor   \quad \mathrm{for} \,  n \geq 1.
\end{equation}
Note that both $(\eta^{A}_{t})_{t\geq 0}$ ,$(\eta^{B}_{t})_{t\geq 0}$ are ASEPs with shifted step initial data,
hence Theorem \ref{TWT} applies to them.
Furthermore, it is easy to see that under the basic coupling, we always have $x_{n}(t)\leq x_{n}^{A}(t),H_{n}(t)\geq H_{n}^{B}(t), n\geq 1. $
\section{Countable state space ASEPs}\label{countable}
Here we collect the bounds of particle positions in countable state space ASEPs that were obtained in \cite{N19}. Unlike in \cite{N19} however,
these bounds  are not directly applicable.   
We amend them to include statements about the rightmost hole in countable state space ASEPs. 
In Section \ref{5},  we will then squeeze the second class particle  in between the leftmost particle and the rightmost hole of countable state space ASEPs.

We  will need the following  bound on the position of the leftmost particle of $\eta^{-\mathrm{step}(Z)}_{\ell}$ .

\begin{prop}[Proposition 3.1 in \cite{N19}]\label{blockp}
Consider  ASEP with reversed step  initial data $x_{-n}^{\mathrm{-step(Z)}}(0)=n+Z,n\geq 0$ and let $\delta>0$. Then  there is a $t_0$ such that for $t>t_{0},R\in \Z_{\geq 1}$ and constants $C_{1},C_{2}$ (which depend on $p$) we have
\begin{align*}
&\Pb\left(  x_{0}^{\mathrm{-step}(Z)}(t)<Z-R\right)\leq C_{1}e^{-C_{2}R}
\\&
\Pb\left( \inf_{0\leq \ell \leq t} x_{0}^{\mathrm{-step}(Z)}(\ell)<Z-t^{\delta}\right)\leq C_{1}e^{-C_{2}t^{\delta}}.
\end{align*}
\end{prop}

Recalling the state space \eqref{omega}, we note that there is a  partial order on $\Omega_{Z}:$ We define
\begin{equation}\label{order}
\eta^{\prime}\preceq\eta^{\prime\prime}\iff \sum_{j=r}^{\infty}1-\eta^{\prime\prime }(j)\leq \sum_{j=r}^{\infty}1-\eta^{\prime}(j)\quad \mathrm{for \, all\,}r\in \Z.
\end{equation}
The following Lemma, taken from \cite{N19} and amended to contain a statement about holes, will be used  to bound the position of the  leftmost particle and the rightmost hole of ASEPs in $\Omega_{Z}.$

\begin{lem}[Lemma 3.1 in \cite{N19}]\label{lem}
Let   $\eta^{\prime},\eta^{\prime\prime}\in \Omega_{Z}$ and consider the basic coupling of two ASEPs $(\eta_{\ell}^{\prime})_{\ell\geq 0},(\eta^{\prime\prime}_{\ell})_{\ell\geq 0}$ started from  $\eta_{0}^{\prime}=\eta^{\prime},\eta^{\prime\prime}_{0}=\eta^{\prime\prime}$.   For $s\geq 0,$ denote   by $x_{0}^{\prime}(s),x_{0}^{\prime\prime}(s)$ the position of the leftmost particle of $\eta_{s}^{\prime},\eta^{\prime\prime}_{s},$  and by $H_{0}^{\prime}(s),H_{0}^{\prime\prime}(s)$ the  position of the rightmost hole of $\eta_{s}^{\prime},\eta^{\prime\prime}_{s}.$
Then, if $\eta^{\prime}\preceq\eta^{\prime\prime}$, we have $x_{0}^{\prime}(s)\leq x_{0}^{\prime\prime}(s)$ and $H_{0}^{\prime}(s)\geq H_{0}^{\prime\prime}(s).$
\end{lem}

Finally, we will use the following proposition, which bounds the position of the leftmost particle/rightmost hole in a certain  countable state space  ASEP.   Using the results on mixing times of \cite{BBHM}, we know the leading order of the time   it takes this ASEP  to reach the reversed step initial data and 
once ASEP has hit the reversed step initial data, we can use Proposition \ref{blockp} to bound the position of the leftmost particle/rightmost hole in ASEP. This leads to the following.

\begin{prop}[Proposition 3.3 in \cite{N19}]\label{DOIT}
Let $a,b,N\in\Z$ and $a\leq b\leq N.$
Consider the ASEP $(\eta^{a,b,N}_{\ell})_{\ell\geq 0}$ with initial data 
\begin{equation}\label{abN}
\Omega_{N-b+a}\ni\eta^{a,b,N}_{0}=\mathbf{1}_{\{a,\ldots,b\}}+\mathbf{1}_{\Z_{\geq N+1}}
\end{equation}
and denote by $x_{0}^{a,b,N}(s)$ the position of the leftmost particle of  $\eta^{a,b,N}_{s}$.
Let $\mathcal{M}=\max\{b-a+1,N-b\}$ and $\varepsilon>0$. Then there are constants $C_{1},C_{2}$ (depending on $p$) and a constant $K$ (depending on $p,\varepsilon$)
so that for $s>K\mathcal{M}$ and $R\in \Z_{\geq 1}$ 
\begin{equation*}
\Pb\left(x_{0}^{a,b,N}(s)<N-b+a-R\right)\leq \frac{\varepsilon}{\mathcal{M}}+ C_{1}e^{-C_{2}R}.
\end{equation*}
Likewise, denoting by $H_{0}^{a,b,N}(s)$ the position of the rightmost hole of  $\eta^{a,b,N}_{s},$ we have
\begin{equation*}
\Pb\left(H_{0}^{a,b,N}(s)>N-b+a+R\right)\leq \frac{\varepsilon}{\mathcal{M}}+ C_{1}e^{-C_{2}R}.
\end{equation*}
\end{prop}
We note that after \cite{BBHM}, the paper  \cite{LL19} proved the cutoff phenomenon for ASEP/biased card shuffling, and in particular obtained the exact asymptotics for the mixing time.
However, the results of  \cite{BBHM} are sufficient for our purposes.

\section{Limit laws for $\mathcal{H},\mathcal{P}$ and $\mathcal{H}-\mathcal{P}$: Proof of Theorem \ref{Theorem3} }\label{4}
We start by computing the limit law for  $\mathcal{H},\mathcal{P}.$ The main step in the proof is to show that we may replace $x_{n},H_{n}$ by $x_{n}^{A},H_{n}^{B}$
from Section \ref{compare}.
\begin{prop}\label{convp}
Let $\mathcal{H},\mathcal{P}$ be defined as in \eqref{PH}.
We have for  $L\in \Z_{\geq 1}$
\begin{align*}
&\lim_{t\to \infty}\Pb(\mathcal{H}=L)=\lim_{t\to \infty}\Pb(\mathcal{P}=L)=F_{L,p}(C)-F_{L+1,p}(C)
\\&\lim_{t\to \infty}\Pb(\mathcal{H}<0)=\lim_{t\to \infty}\Pb(\mathcal{P}<0)=0
\\&\lim_{t\to \infty}\Pb(\mathcal{H}=0)=\lim_{t\to \infty}\Pb(\mathcal{P}=0)=1-F_{1,p}(C).
\end{align*}
\end{prop}
\begin{proof}

Throughout the proof, we assume that all appearing ASEPs are coupled through the basic coupling. First we consider $L\geq1$. Let us deal with $\mathcal{P}$ first. 
Recall the ASEP \eqref{A}, and define

\begin{equation}\label{PA}
\begin{aligned}
&\mathcal{P}^{A}=\sup\{i\in \Z| x_{i}^{A}(t-t^{\chi})>-t^{\chi^{\prime}}\}.
\end{aligned}
\end{equation}
We have that
\begin{align*}
&\{\mathcal{P}^{A}=L\}=\{x_{L}^{A}(t-t^{\chi})>-t^{\chi^{\prime}}\}\cap\{x_{L+1}^{A}(t-t^{\chi})\leq -t^{\chi^{\prime}}\}
\\& \{\mathcal{P}=L\}=\{x_{L}(t-t^{\chi})>-t^{\chi^{\prime}}\}\cap\{x_{L+1}(t-t^{\chi})\leq -t^{\chi^{\prime}}\}.
\end{align*}
Since $x_{L+1}^{A}(t-t^{\chi})< x_{L}^{A}(t-t^{\chi})$ it follows  that 
\begin{equation*}
\lim_{t\to \infty}\Pb(\mathcal{P}^{A}=L)=\lim_{t\to\infty}\Pb(x_{L}^{A}(t-t^{\chi})>-t^{\chi^{\prime}})-\Pb(x_{L+1}^{A}(t-t^{\chi})>-t^{\chi^{\prime}}),
\end{equation*}
and as $\chi<\chi^{\prime}<1/2$ we can conclude from Theorem \ref{TWT} that
\begin{equation*}
\lim_{t\to \infty}\Pb(\mathcal{P}^{A}=L)= F_{L,p}(C)-F_{L+1,p}(C).
\end{equation*}
Thus to finish the proof, we may show that 
\begin{equation}\label{0+0}
\begin{aligned}
\lim_{t\to \infty}\Pb(\{\mathcal{P}^{A}=L\}\setminus \{\mathcal{P}=L\})+\Pb(\{\mathcal{P}=L\}\setminus \{\mathcal{P}^{A}=L\})=0.
\end{aligned}
\end{equation}
Since $x_{i}^{A}(t-t^{\chi})\geq x_{i}(t-t^{\chi}) $ for all $i\geq 1$, we have that
\begin{equation}\begin{aligned}\label{dsds}
& \{x_{L}(t-t^{\chi})>-t^{\chi^{\prime}}\}\subseteq \{x_{L}^{A}(t-t^{\chi})>-t^{\chi^{\prime}}\}
\\&\{x_{L}^{A}(t-t^{\chi})\leq-t^{\chi^{\prime}}\}\subseteq \{x_{L}(t-t^{\chi})\leq-t^{\chi^{\prime}}\}.
\end{aligned}\end{equation}
Because of \eqref{dsds}, in order to show \eqref{0+0}, it suffices to show
\begin{equation}\label{0=0}
\begin{aligned}
&\lim_{t\to \infty}\Pb(\{x_{L}^{A}(t-t^{\chi})>-t^{\chi^{\prime}}\}\setminus \{x_{L}(t-t^{\chi})>-t^{\chi^{\prime}}\})=0
\\&\lim_{t\to \infty}\Pb(\{x_{L+1}(t-t^{\chi})\leq-t^{\chi^{\prime}}\}\setminus \{x_{L+1}^{A}(t-t^{\chi})\leq -t^{\chi^{\prime}}\})=0.
\end{aligned}
\end{equation}
Now both statements of \eqref{0=0} follow from showing 
\begin{equation*}
\lim_{t\to\infty}\Pb(\{x_{i}^{A}(t-t^{\chi})\neq x_{i}(t-t^{\chi})\}\cap \{x_{i}(t-t^{\chi})\leq-t^{\chi^{\prime}}\})=0
\end{equation*}
for arbitrary $i\in \Z_{\geq 1}$. Define the random times 
\begin{equation*}
\tau_{i}=\inf \{ \ell |\,x_{i}^{A}(\ell)\neq  x_{i}(\ell)\}.
\end{equation*}
Then it suffices to show
\begin{equation}
\lim_{t\to\infty}\Pb(\{ \tau_{i}\leq t-t^{\chi})\} \cap \{x_{i}(t-t^{\chi})\leq-t^{\chi^{\prime}}\})=0.
\end{equation}
Now it follows immediately from the bound (56) of \cite{N19}, that in fact for any $\delta>0$ we have
\begin{equation}\label{56}
\lim_{t\to\infty}\Pb(\{ \tau_{i}\leq t-t^{\chi})\} \cap \{x_{i}(t-t^{\chi})\leq-(i+1)t^{\delta/2}\})=0.
\end{equation}
The only inessential difference between the bound (56) of \cite{N19} and \eqref{56}  is that in  \cite{N19}, we look at time $t$  (rather than $t-t^{\chi}$) and   the ASEP  $\eta^{1}$ is considered (i.e. the second class particle is replaced by a particle).
Since \eqref{56} implies \eqref{0=0} which implies \eqref{0+0}, we conclude
\begin{equation*}
\lim_{t\to \infty}\Pb(\mathcal{P}=L)= F_{L,p}(C)-F_{L+1,p}(C).
\end{equation*}

Next we treat $\Pb(\mathcal{P}<0)$. Note that we may bound
\begin{align*}
\Pb(\mathcal{P}<0)\leq \Pb(x_{0}(t-t^{\chi})\leq-t^{\chi^{\prime}})\leq  \Pb(x_{0}^{-\mathrm{step}(1)}(t-t^{\chi})\leq-t^{\chi^{\prime}}),
\end{align*}
so we obtain from Proposition \ref{blockp} that $\lim_{t\to\infty}\Pb(\mathcal{P}<0)=0.$

Finally, to treat $L=0$, note that because $\lim_{t\to\infty} \Pb(x_{0}(t-t^{\chi})>-t^{\chi^{\prime}})=1,$ we  get
 \begin{equation}\label{23}
 \lim_{t\to \infty}\Pb(\mathcal{P}=0)=\lim_{t\to\infty}\Pb(x_{0}(t-t^{\chi})>-t^{\chi^{\prime}})-\Pb(x_{1}(t-t^{\chi})>-t^{\chi^{\prime}})=1-F_{1,p}(C),
 \end{equation}
 where we used that we may replace $x_{1}(t-t^{\chi})$ by   $x_{1}^{A}(t-t^{\chi})$ in \eqref{23} because of \eqref{56}.
 
 The statements for $\mathcal{H}$ can be proven similarly : One introduces  
\begin{equation}\label{HB}
\begin{aligned}
&\mathcal{H}^{B}=\sup\{i\in \Z| H_{i}^{B}(t-t^{\chi})<t^{\chi^{\prime}}\},
\end{aligned}
\end{equation}
and uses the particle-hole duality.
\end{proof}

A corollary, we can show that $-\mathcal{H},-\mathcal{P}$ converge to $F_\GUE$ in a double limit under KPZ scaling.

\begin{cor} \label{corgue}
Let $C=C(M)$ be as in \eqref{C}  and let $s \in \R$.
Then 
\begin{align*}
&\lim_{M\to \infty}\lim_{t \to \infty}\Pb\left(\frac{\mathcal{H}-M}{M^{1/3}}\leq s\right)
\\&=\lim_{M\to \infty}\lim_{t \to \infty}\Pb\left(\frac{\mathcal{P}-M}{M^{1/3}}\leq s\right)=1-F_{\GUE}(-s).
\end{align*}
\end{cor}
\begin{proof}
By Proposition \ref{convp}, it  suffices to consider $\mathcal{H}$ and we have that 
\begin{equation}
\lim_{t \to \infty}\Pb\left(\frac{\mathcal{H}-M}{M^{1/3}}\leq s\right)=1-F_{\lfloor M+M^{1/3}s \rfloor+1,p}(C(M)).
\end{equation}
Upon setting $\tilde{M}=\lfloor M+M^{1/3}s \rfloor+1$ we get that \begin{equation*}C(M)=\frac{2\sqrt{\tilde{M}}-s\tilde{M}^{-1/6}}{\sqrt{p-q}}+o(\tilde{M}^{-1/6}).\end{equation*}
The statement follows now from Proposition \ref{FGUE}.
\end{proof}
Finally, we show that $\mathcal{H},\mathcal{P}$ are independent asymptotically. The first step in the proof is again to replace $x_{n}, H_{n}$ by $x_{n}^{A}, H_{n}^{B} $ and
then apply slow decorrelation to the $x_{n}^{A}$ to make sure  that $x_{n}^{A}$ and $H_{n}^{B}$ stay in disjoint space-time regions, hence they are asymptotically  independent.
\begin{prop}\label{indep}
We have for $R,L\in \Z$
\begin{equation*}
\lim_{t\to \infty}\Pb\left(\{\mathcal{H}=R\}\cap \{\mathcal{P}=L\}\right)=\lim_{t\to \infty}\Pb(\mathcal{H}=R) \Pb(\mathcal{P}=L).
\end{equation*}
\end{prop}
\begin{proof}
We assume that all appearing ASEPs are coupled via the basic coupling.
We may assume $R,L \geq 0$ because otherwise the statement is trivial due to Proposition \ref{convp}.
Recall $ \mathcal{H}^{B}$ from \eqref{HB} and $\mathcal{P}^{A}$ from \eqref{PA}.
Then it follows from \eqref{0+0} and the analogous statement for $\mathcal{H},\mathcal{H}^{B}$  that 
\begin{equation*}
\begin{aligned}
\lim_{t\to \infty}\Pb\left(\{\mathcal{H}=R\}\cap \{\mathcal{P}=L\}\right)=\lim_{t\to \infty}\Pb\left(\{\mathcal{H}^{B}=R\}\cap \{\mathcal{P}^{A}=L\}\right).
\end{aligned}
\end{equation*}

For brevity, let us define \begin{align*}&A_{1,L}=\{x_{L}^{A}(t-t^{\chi})>-t^{\chi^{\prime}}\},\\&A_{2,L}=\{x_{L+1}^{A}(t-t^{\chi})\leq-t^{\chi^{\prime}}\}.\end{align*}

Let $\kappa\in (1/2,1)$  and define the events
\begin{equation*}
\begin{aligned}
&A_{3,L}=\{x_{L}^{A}(t-t^{\chi}-t^{\kappa})+(p-q)t^{\kappa}>-t^{\chi^{\prime}}\},\\& A_{4,L}= \{x_{L+1}^{A}(t-t^{\chi}-t^{\kappa})+(p-q)t^{\kappa}\leq-t^{\chi^{\prime}}\}, \\&I_{L}=A_{3,L}\cap A_{4,L}.
\end{aligned}
\end{equation*}

We will  prove
\begin{equation}\label{0+0+0}
\begin{aligned}
\lim_{t\to \infty}\Pb(\{\mathcal{P}^{A}=L\}\setminus I_{L})+\Pb(I_{L}\setminus \{\mathcal{P}^{A}=L\})=0.
\end{aligned}
\end{equation}

To prove \eqref{0+0+0}, it suffices to show 
\begin{equation}\label{4sum}
\lim_{t\to\infty}\Pb(A_{1,L}\setminus A_{3,L})+\Pb(A_{3,L}\setminus A_{1,L})+\Pb(A_{2,L}\setminus A_{4,L})+\Pb(A_{4,L}\setminus A_{2,L})=0.
\end{equation}

By the slow-decorrelation statement (Proposition 4.2 in \cite{N19} with $t=t-t^{\chi}$) we have that for $\varepsilon>0$

\begin{equation}\label{slowdec}
\lim_{t\to\infty}\Pb(|x_{L}^{A}(t-t^{\chi}-t^{\kappa})+(p-q)t^{\kappa}-x_{L}^{A}(t-t^{\chi})|\geq \varepsilon t^{1/2})=0.
\end{equation}
Now, using \eqref{slowdec}, we  have for any $\varepsilon>0$
\begin{equation*}
\begin{aligned}
\lim_{t\to\infty}\Pb(A_{1,L}\setminus A_{3,L})&\leq \lim_{t\to \infty}\Pb(x_{L}^{A}(t-t^{\chi}-t^{\kappa})+(p-q)t^{\kappa}\in [-t^{\chi^{\prime}}-\varepsilon t^{1/2},-t^{\chi^{\prime}}])\\&=F_{L,p}(C+\varepsilon)-F_{L,p}(C).
\end{aligned}
\end{equation*}
Sending $\varepsilon \to 0$ shows the desired convergence to zero. The other three limits in \eqref{4sum} are completely analogous.

We have thus shown that 
\begin{equation}\label{stage}
\lim_{t\to \infty}\Pb\left(\{\mathcal{H}=R\}\cap \{\mathcal{P}=L\}\right)=\lim_{t \to \infty}\Pb\left(\{\mathcal{H}^{B}=R\}\cap I_{L} \right).
\end{equation}

The final step is to construct random variables \begin{equation}\label{tilderv}\tilde{H}^{B}_{R}(t-t^{\chi}),\tilde{H}^{B}_{R+1}(t-t^{\chi}), \tilde{x}_{L}^{A}(t-t^{\chi}-t^{\kappa}),\tilde{x}_{L+1}^{A}(t-t^{\chi}-t^{\kappa})\end{equation}
such that   $\tilde{H}^{B}_{R}(t-t^{\chi}),\tilde{H}^{B}_{R+1}(t-t^{\chi})$ are independent from $ \tilde{x}_{L}^{A}(t-t^{\chi}-t^{\kappa}),\tilde{x}_{L+1}^{A}(t-t^{\chi}-t^{\kappa})$
and  having a tilde or not makes no difference asymptotically :
\begin{equation}
\begin{aligned}\label{tilde}
&\lim_{t\to \infty}\Pb(\tilde{H}^{B}_{R}(t-t^{\chi})\neq H^{B}_{R}(t-t^{\chi}))+\Pb(\tilde{H}^{B}_{R+1}(t-t^{\chi})\neq H^{B}_{R+1}(t-t^{\chi}))=0
\\&\lim_{t\to\infty}\Pb(\tilde{x}^{B}_{L}(t-t^{\chi}-t^{\kappa})\neq x^{B}_{L}(t-t^{\chi}-t^{\kappa}))=0\\&\lim_{t\to\infty}\Pb(\tilde{x}^{B}_{L+1}(t-t^{\chi}-t^{\kappa})\neq x^{B}_{L+1}(t-t^{\chi}-t^{\kappa}))=0.
\end{aligned}
\end{equation}
The construction of the random variables   \eqref{tilderv}
is based on the observation that  $H^{B}_{R}(t-t^{\chi}),H^{B}_{R+1}(t-t^{\chi})$ and $ x_{L}^{A}(t-t^{\chi}-t^{\kappa}),x_{L+1}^{A}(t-t^{\chi}-t^{\kappa})$ stay in disjoint, deterministic  space time-regions with probability 
going to one.  Their construction is done in detail between formulas (41) and (42) of \cite{N19}, the proof of \eqref{tilde} is (42) in \cite{N19}, (with the inessential difference that in \cite{N19}, we have  $t$ instead of  $t-t^{\chi}$).
 Define 
 \begin{equation*}
\begin{aligned}
&\tilde{A}_{3,L}=\{\tilde{x}_{L}^{A}(t-t^{\chi}-t^{\kappa})+(p-q)t^{\kappa}>-t^{\chi^{\prime}}\},\\& \tilde{A}_{4,L}= \{\tilde{x}_{L+1}^{A}(t-t^{\chi}-t^{\kappa})+(p-q)t^{\kappa}\leq-t^{\chi^{\prime}}\}, \\&\tilde{I}_{L}=\tilde{A}_{3,L}\cap \tilde{A}_{4,L},
\\& \tilde{J}_{R} =\{\tilde{H}^{B}_{R}(t-t^{\chi})<t^{\chi^{\prime}}\}   \cap \{\tilde{H}^{B}_{R+1}(t-t^{\chi})\geq t^{\chi^{\prime}}\} .
\end{aligned}
\end{equation*}

 We can thus conclude by computing 
 \begin{align*}
 \lim_{t \to \infty}\Pb\left(\{\mathcal{H}^{B}=R\}\cap I_{L} \right)&=\lim_{t \to \infty}\Pb\left(   \tilde{J}_{R}  \cap \tilde{I}_{L} \right)
 \\&=\lim_{t \to \infty}\Pb(   \tilde{J}_{R}  )\Pb( \tilde{I}_{L} )
  \\&=\lim_{t\to \infty}\Pb(\mathcal{H}=R) \Pb(\mathcal{P}=L),
 \end{align*}
 where for the last identity we first used \eqref{tilde} -thus removing the $\tilde{}$\, -,  then  \eqref{0+0+0} -thus replacing $x_{L}^{A}(t-t^{\chi}-t^{\kappa})+(p-q)t^{\kappa}$ by $x_{L}^{A}(t-t^{\chi})$ - and finally  \eqref{0+0}\footnote{and the analogue of \eqref{0+0}  for $\mathcal{H},\mathcal{H}^{B}$} -thus replacing $\mathcal{H}^{B},\mathcal{P}^{A}$ by $\mathcal{H},\mathcal{P} $. 
 By \eqref{stage}, this finishes the proof.

\end{proof}

\begin{proof}[Proof of Theorem \ref{Theorem3}]
It follows from Proposition \ref{convp} and \ref{indep} that  the pair $\left(\frac{\mathcal{H}-M}{M^{1/3}},  \frac{\mathcal{P}-M}{M^{1/3}}\right)$ converges, as $t \to \infty,$ in distribution to the product measure $\mu^{M}\otimes \mu^{M}$ on $M^{-1/3}\Z_{\geq-M}\times \Z_{\geq -M}$  given by
$\mu^{M}(\{M^{-1/3}i\})=F_{M+i,p}(C)-F_{M+i+1,p}(C)$ with $F_{0,p}(C):=1$. By Corollary \ref{corgue}, for $C=C(M)$,  $\mu^{M}\otimes \mu^{M}$ converges, as $M\to \infty,$ in distribution to the product measure  $\mu^{-\GUE}\otimes \mu^{-\GUE}$ on $\R^{2}$ with $ \mu^{-\GUE}((-\infty,\xi]):=1-F_{\GUE}(-\xi)$ for $\xi \in \R$. This  implies the result by the continuous mapping theorem.
\end{proof}

\section{Proof of Theorem \ref{couplthm}: $|X(t)-\mathcal{H}+\mathcal{P}|M^{-\varepsilon}$ goes to zero }\label{5}
Let us start by recalling that  with $L,R\in \Z$
\begin{align*}
&\{\mathcal{H}=R\}=\{H_{R}(t-t^{\chi})<t^{\chi^{\prime}}\}\cap\{H_{R+1}(t-t^{\chi})\geq t^{\chi^{\prime}}\}
\\& \{\mathcal{P}=L\}=\{x_{L}(t-t^{\chi})>-t^{\chi^{\prime}}\}\cap\{x_{L+1}(t-t^{\chi})\leq -t^{\chi^{\prime}}\}.
\end{align*}
We define for $\delta\in (0,\chi)$   the events 
\begin{align*}
&B_L=\{x_{L}(t-t^{\chi})>-t^{\delta}\}\cap\{x_{L+1}(t-t^{\chi})\leq -t^{\chi^{\prime}}\}
\\& D_R=\{H_{R}(t-t^{\chi})<t^{\delta}\}\cap\{H_{R+1}(t-t^{\chi})\geq t^{\chi^{\prime}}\}.
\end{align*}
Since $\delta<\chi<\chi^{\prime},$ clearly we have
\begin{equation}\label{incl}
B_L \subseteq \{\mathcal{P}=L\} \quad 
D_R \subseteq \{\mathcal{H}=R\}.
\end{equation}
The strategy of the proof of Theorem \ref{couplthm} is as follows (see  Figure \ref{Graph} for an illustration):   We partition the probability space according to \begin{equation*}\bigcup_{L,R\in \Z}\{\mathcal{H}=R\}\cap \{\mathcal{P}=L\}.\end{equation*} It actually suffices to  consider $L,R\in \{0,\ldots,2M\}.$ We first show in Proposition  \ref{redprop}, that we may replace $\{\mathcal{H}=R\}\cap \{\mathcal{P}=L\}$
by $B_{L}\cap D_{R}$. On $B_{L}\cap D_{R},$ we then show that all the particles and holes  $x_{L+n},H_{R+n}, n \geq 1$ are irrelevant  asymptotically.
Hence we may replace $X(t)$ by a second class particle $\tilde{X}(t)$ which lives in an ASEP on $\Omega_{R-L}$. This statement appears as Proposition \ref{Xtild}.
Finally, in Proposition \ref{expprop}, we show that $X(t)$  is exponentially unlikely to be far away from $R-L$ because this is true for $\tilde{X}(t)$ .
 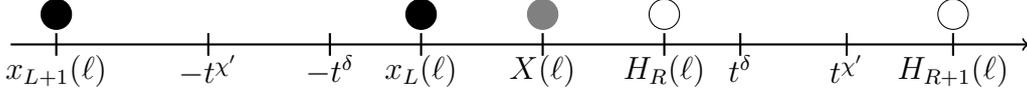
\begin{figure}\begin{center}
\begin{tikzpicture}[scale=2]

  \draw[thick, ->] (-1,0) -- (5.7,0);

  \filldraw[gray](2.5,0.2) circle (0.1);
    \draw (2.5,0) node[below] {$X(\ell)$};
  \filldraw (1.7,0.2) circle (0.1);
      \draw (1.7,0) node[below] {$x_{L}(\ell)$};
    \filldraw (-0.7,0.2) circle (0.1);
           \draw (-0.7,0) node[below] {$x_{L+1}(\ell)$};
\draw (3.3,0.2) circle (0.1);
     \draw (3.3,0) node[below] {$H_{R}(\ell)$};
   \draw (3.8,0) node[below] {$t^{\delta}$};
      \draw (1.1,0) node[below] {$-t^{\delta}$};
          \draw (4.5,0) node[below] {$t^{\chi^{\prime}}$};
            \draw (0.3,0) node[below] {$-t^{\chi^{\prime}}$};
    \draw (5.2,0.2) circle (0.1);
       \draw (5.2,0) node[below] {$H_{R+1}(\ell)$};
       \foreach \x in {5.2,2.5,1.7,-0.7,3.3,3.8,1.1,4.5,0.3}
       \draw[thick] (\x,0.075)--(\x,-0.075);
    
\end{tikzpicture}\end{center}
\caption{The particle configuration $\eta_{\ell}$ at time $\ell=t-t^{\chi}$  on the event $B_{L}\cap D_{R}\cap\{|X(\ell)|\leq t^{\delta}\}$:
The second class particle is gray, first class particles are black, holes are white.
 The hole $H_{R},$ particle $x_{L}$ and the second class particle lie in $\{-t^{\delta},\ldots, t^{\delta}\}$.
All particles and holes $x_{L+n},H_{R+n},n\geq1,$ lie outside $\{-t^{\chi^{\prime}},\ldots,t^{\chi^{\prime}}\}$. Since $\chi^{\prime}>\chi,$ the holes/particles  $x_{L+n},H_{R+n},n\geq1,$
will not interact with the second class particle during $[t-t^{\chi},t]$ as they are too far away from it.
Consequently, we may replace all particles $x_{L+n},n\geq1,$ by holes and all holes $H_{R+n},n\geq1,$ by particles. 
Since $\delta<\chi,$ the second class particle in the new configuration has enough time during $[t-t^{\chi},t]$  to be, at time $t$,  close  to its equilibrium position    $R-L+o(M^{\varepsilon}).$ 
  }
\label{Graph}
\end{figure}

We start by replacing $\{\mathcal{H}=R\}\cap \{\mathcal{P}=L\}$ by  $B_{L}\cap D_{R}$ in the following proposition.
\begin{prop}\label{redprop} We have
\begin{align*}
&\lim_{M\to \infty}\lim_{t\to\infty}\Pb\left( |X(t)-\mathcal{H}+\mathcal{P}|>M^{\varepsilon}\right)\\&=\lim_{M\to \infty}\lim_{t\to\infty}\sum_{L,R\in \{0,\ldots,2M\}}\Pb\left(\{|X(t)-R+L|\geq M^{\varepsilon}\}\cap B_{L}\cap D_{R}\right).
\end{align*}
\end{prop}
\begin{proof}
We first note that 
\begin{align}
&\lim_{M\to \infty}\lim_{t\to\infty}\Pb\left( |X(t)-\mathcal{H}+\mathcal{P}|>M^{\varepsilon}\right)\nonumber\\&=\lim_{M\to \infty}\lim_{t\to\infty}\sum_{L,R\in \Z}\Pb\left(\{|X(t)-R+L|\geq M^{\varepsilon}\}\cap \{\mathcal{H}=R\}\cap \{\mathcal{P}=L\}\right)
\nonumber
\\&=\label{lastline}\lim_{M\to \infty}\lim_{t\to\infty}\sum_{L,R\in \{0,\ldots,2M\}}\Pb\left(\{|X(t)-R+L|\geq M^{\varepsilon}\}\cap \{\mathcal{H}=R\}\cap \{\mathcal{P}=L\}\right).
\end{align}
Here, for the identity \eqref{lastline} we used that by Proposition  \ref{convp}, 
\begin{align*}
\lim_{t\to \infty}\Pb(\mathcal{H}<0)=\lim_{t\to \infty}\Pb(\mathcal{P}<0)=0
\end{align*}
and that by Corollary \ref{corgue}
\begin{align*}
\lim_{M\to \infty}\lim_{t\to \infty}\Pb(\mathcal{H}>2M)=\lim_{M\to \infty}\lim_{t\to \infty}\Pb(\mathcal{P}>2M)=0.
\end{align*}

Next define the events
\begin{align*}
&B_{L}^{A}=\{x_{L}^{A}(t-t^{\chi})>-t^{\delta}\}\cap\{x_{L+1}^{A}(t-t^{\chi})\leq -t^{\chi^{\prime}}\}
\\& D_{R}^{B}=\{H_{R}^{B}(t-t^{\chi})<t^{\delta}\}\cap\{H_{R+1}^{B}(t-t^{\chi})\geq t^{\chi^{\prime}}\}.
\end{align*}
It follows easily from \eqref{0+0} and \eqref{56} that 
\begin{align*}
\lim_{t\to \infty} \Pb( \{\mathcal{P}=L\}\setminus B_{L})&=\lim_{t\to \infty}\Pb( \{\mathcal{P}^{A}=L\}\setminus B_{L}^{A})
\end{align*}
from which we deduce using Theorem \ref{TWT}
\begin{equation}
\lim_{t\to \infty} \Pb( \{\mathcal{P}=L\}\setminus B_{L})\leq \lim_{t\to \infty}\Pb(x_{L}^{A}(t-t^{\chi})\in[-t^{\chi^{\prime}},-t^{\delta}])=0.
\end{equation}
Likewise, we obtain 
\begin{align*}
\lim_{t\to \infty} \Pb( \{\mathcal{H}=R\}\setminus D_{R})=0.
\end{align*}
The inclusions \eqref{incl} thus imply that 
\begin{align*}
&\lim_{t\to\infty}\sum_{L,R\in \{0,\ldots,2M\}}\Pb\left(\{|X(t)-R+L|\geq M^{\varepsilon}\}\cap \{\mathcal{H}=R\}\cap \{\mathcal{P}=L\}\right)
\\&=\lim_{t\to\infty}\sum_{L,R\in \{0,\ldots,2M\}}\Pb\left(\{|X(t)-R+L|\geq M^{\varepsilon}\}\cap B_{L}\cap D_{R}\right).
\end{align*}
\end{proof}

In order to be able to replace $X(t)$ by the second class particle $\tilde{X}(t)$ defined in \eqref{X(t)2} below, we need to know that $X(t-t^{\chi})$ is not far away from the origin. 
\begin{prop}\label{deltaprp}
We have that for any $0<\delta< \chi <1/2$ 
\begin{equation*}
\lim_{t\to \infty}\Pb(|X(t-t^{\chi})|>t^{\delta})=0.
\end{equation*}
\end{prop}
\begin{proof}
We start by showing
\begin{equation}\label{jaja}
\lim_{t\to \infty}\Pb(X(t-t^{\chi})<-t^{\delta})=0.
\end{equation}

Define for all $s\geq 0$ a $\Z-$valued random  variable $\mathcal{X}^{P}(s)$ via 
\begin{equation*}
\mathcal{X}^{P}(s)=R \iff  x_{R}^{1}(s)=X(s), \, R\in \Z.
\end{equation*}
This means, $\mathcal{X}^{P}(s)$ is the label of the particle in $ \eta^{1} $  (defined in \eqref{IC1}) that is at position $X(s).$
We will show separately
\begin{align}\label{first}
&\lim_{t\to \infty}\Pb(\{\mathcal{X}^{P}(t-t^{\chi})\leq 0\}\cap \{X(t-t^{\chi})<-t^{\delta}\})=0
\\&\label{second} \lim_{t\to \infty}\Pb(\{\mathcal{X}^{P}(t-t^{\chi})\in\{1,\ldots, \lfloor t^{\delta/4}\rfloor\}\}\cap \{X(t-t^{\chi})<-t^{\delta}\})=0
\\& \label{third} \lim_{t\to \infty}\Pb(\{\mathcal{X}^{P}(t-t^{\chi})\geq t^{\delta/4}\}\cap \{X(t-t^{\chi})<-t^{\delta}\})=0.
\end{align}

For \eqref{first}, we have that 
\begin{align*}
&\lim_{t\to \infty}\Pb(\{\mathcal{X}^{P}(t-t^{\chi})\leq 0\}\cap \{X(t-t^{\chi})<-t^{\delta}\})
\\&\leq \lim_{t\to \infty}\Pb(x_{0}^{-\mathrm{step}}(t-t^{\chi})<-t^{\delta})=0.
\end{align*}
To see \eqref{second}, define $\mathcal{T}_{0}=0$ and for $i\in \Z_{\geq0}$
\begin{equation*}
\mathcal{T}_{i+1}=\inf \{\ell| \ell\geq \mathcal{T}_{i},x_{i+1}^{1}(\ell)=x_{i}^{1}(\ell)-1\}
\end{equation*}  
and define the event 
\begin{equation*}
\mathcal{E}_{s}^{R}= \{\mathcal{T}_{0}<\mathcal{T}_{1}<\cdots < \mathcal{T}_{R}\leq s\}.
\end{equation*}  
  Now  if  $\mathcal{X}^{P}(\cdot)$ makes a jump from $i$ to $i+1$ at time $\tilde{\mathcal{T}}_{i},$  (i.e. 
$\mathcal{X}^{P}(\tilde{\mathcal{T}}_{i}^{-})=i, \mathcal{X}^{P}(\tilde{\mathcal{T}}_{i})=i+1$), then necessarily $x_{i+1}^{1}(\tilde{\mathcal{T}}_{i})=x_{i}^{1}(\tilde{\mathcal{T}}_{i})-1$.
From this and $\mathcal{X}^{P}(0)=0$  we derive for $R\in \Z_{\geq 1}$ that 
\begin{equation*}
\{\mathcal{X}^{P}(s)=R\}\subseteq \mathcal{E}_{s}^{R}.
\end{equation*}
With a proof that is a trivial adaption  of the proof of (56) in \cite{N19}, we obtain 
\begin{equation}\label{562}
\Pb( \mathcal{E}_{t-t^{\chi}}^{R}\cap\{x_{R}^{1}(t-t^{\chi})\leq -(R+1)t^{\delta/2}\} )  \leq (R+1)C_{1}e^{-C_{2}t^{\delta/2}}.
\end{equation}
Consequently, we may bound
\begin{align*}
& \sum_{R \in \{1,\ldots, \lfloor t^{\delta/4}\rfloor\}} \Pb(\{\mathcal{X}^{P}(t-t^{\chi})=R\}\cap \{X(t-t^{\chi})<-t^{\delta}\})
 \\& \leq \sum_{R \in \{1,\ldots, \lfloor t^{\delta/4}\rfloor\}} \Pb( \mathcal{E}_{t-t^{\chi}}^{R}\cap\{x_{R}^{1}(t-t^{\chi})\leq -(R+1)t^{\delta/2}\} )
 \\&\leq C_{1}e^{-C_{2}t^{\delta/2}}( t^{\delta/4}+1)t^{\delta/4}/2 \to_{t\to\infty}0,
\end{align*}
proving \eqref{second}.

 Finally, to prove \eqref{third}, we note for $R\geq 1$ 
 \begin{equation*}
\{\mathcal{X}^{P}(s)\geq R\}\subseteq \{\mathcal{X}^{P}(\ell )= R,\mathrm{\, for\, an\,}\ell\in[0,s]\}\subseteq\mathcal{E}_{s}^{R}.
\end{equation*}
Thus we have
\begin{align*}
&\Pb(\{\mathcal{X}^{P}(t-t^{\chi})\geq t^{\delta/4}\}\cap \{X(t-t^{\chi})<-t^{\delta}\})
\\&\leq \Pb (\mathcal{E}_{t-t^{\chi}}^{t^{\delta/4}}  )\leq  \Pb (\mathcal{E}_{t-t^{\chi}}^{t^{\delta/4}}  \cap \{x_{t^{\delta/4}}^{1}(t-t^{\chi})<-t^{\delta}\} )+\Pb(x_{t^{\delta/4}}^{1}(t-t^{\chi})\geq -t^{\delta}).
\end{align*}

Now it follows from \eqref{562}
that \begin{equation*}
\lim_{t\to\infty}\Pb (\mathcal{E}_{t-t^{\chi}}^{t^{\delta/4}}  \cap \{x_{t^{\delta/4}}^{1}(t-t^{\chi})<-t^{\delta}\} )=0.
\end{equation*}
Furthermore, we have for any $L\geq 1$ fixed that 
\begin{equation*}
\lim_{t\to\infty}\Pb(x_{t^{\delta/4}}^{1}(t-t^{\chi})\geq -t^{\delta})\leq \lim_{t\to\infty} \Pb(x_{L}^{1}(t-t^{\chi})\geq -t^{\delta})=F_{L,p}(C).
\end{equation*}
Since $\lim_{L\to\infty}F_{L,p}(C)=0,$  we have thus proven \eqref{third} and hence \eqref{jaja}.

Finally, to show

\begin{equation*}
\lim_{t\to \infty}\Pb(X(t-t^{\chi})>t^{\delta})=0,
\end{equation*}
we proceed in a very similar way: We define  for all $s\geq 0$ a $\Z-$valued random  variable $\mathcal{X}^{H}(s)$ via 
\begin{equation*}
\mathcal{X}^{H}(s)=R \iff  H_{R}^{2}(s)=X(s), \, R\in \Z
\end{equation*}
and consider three different cases  as in \eqref{first}, \eqref{second},\eqref{third}.

\end{proof}

For brevity, let us write in the following
\begin{equation}\label{FLR}
\mathcal{F}_{L,R}^{\delta}=B_{L}\cap D_{R} \cap \{   |X(t-t^{\chi})|\leq t^{\delta}\}.
\end{equation}

Furthermore, define 
\begin{equation*}
\tilde{\eta}_{t-t^{\chi}}(j)=\mathbf{1}_{\{|j|\leq t^{\delta}\}}\eta_{t-t^{\chi}}(j)+\mathbf{1}_{\{j>t^{\delta}\}}.
\end{equation*}
Let $(\tilde{\eta}_{\ell}),\ell\geq t-t^{\chi})$ be the ASEP starting  at time $t-t^{\chi}$
 from $\tilde{\eta}_{t-t^{\chi}}$ and  coupled with $(\eta_{\ell},\ell\geq 0)$ 
 via the basic coupling.
On the very likely event $ \{   |X(t-t^{\chi})|\leq t^{\delta}\}$, $\tilde{\eta}_{t-t^{\chi}}$ has a second class particle at position $X(t-t^{\chi})$. 
We will denote by $\tilde{X}(\ell)$ the position of the second class particle of 
$\tilde{\eta}_{\ell}$  for $\ell \in [t-t^{\chi},t]$. If we denote  $G_{2}=\{j\in \Z|\tilde{\eta}_{t-t^\chi}(j)=1\},$  let us write 
\begin{equation*}
\tilde{\eta}^{1}_{t-t^{\chi}}(j)=\min\{\tilde{\eta}_{t-t^\chi}(j),1\} \quad \tilde{\eta}^{2}_{t-t^{\chi}}(j)=\mathbf{1}_{G_{2}}\tilde{\eta}_{t-t^\chi}(j)
\end{equation*} 
(recall that the occupation variable for the second class particle is $2$).
This simply means that in $\tilde{\eta}^{1}_{t-t^{\chi}},$ the second class particle is replaced by a first class particle, whereas in 
$\tilde{\eta}^{2}_{t-t^{\chi}},$ the second class particle is replaced by a hole.
If  $ \{   |X(t-t^{\chi})|\leq t^{\delta}\}$ holds, $\tilde{\eta}^{1}_{t-t^{\chi}}(j)=\tilde{\eta}^{2}_{t-t^{\chi}}(j)=\tilde{\eta}_{t-t^{\chi}}(j)$  for all $j$  except
$j=X(t-t^{\chi})$.
 We can then define using the basic coupling (see Section \ref{lbc})
\begin{equation}\label{X(t)2}
\tilde{X}(t):=\sum_{j\in \Z}j\mathbf{1}_{\{\tilde{\eta}^{1}_{t}(j)\neq \tilde{\eta}^{2}_{t}(j)\}}.
\end{equation}

\begin{prop}\label{Xtild}
We have 
\begin{equation}
\lim_{t\to \infty}\Pb(\mathcal{F}_{L,R}^{\delta} \cap\{X(t)\neq \tilde{X}(t)\})=0.
\end{equation}
\end{prop}
\begin{proof}
Note that 
\begin{equation}\label{guilt}
\mathcal{F}_{L,R}^{\delta} \subseteq \{ \tilde{\eta}_{t-t^{\chi}}(j)=\eta_{t-t^{\chi}}(j), \, \mathrm{ for}  \, |j|\leq t^{\chi^{\prime}}/2\},
\end{equation}
in fact we could replace $ t^{\chi^{\prime}}/2$ by $ t^{\chi^{\prime}}$ in \eqref{guilt}. 

We will show that on $\mathcal{F}_{L,R}^{\delta}$, it is very unlikely  
that $(\tilde{\eta}_{\ell}, \ell \in [t-t^{\chi},t]),$ or $(\eta_{\ell}, \ell \in [t-t^{\chi},t])$ have a jump that involves the sites $\pm  t^{\chi^{\prime}}/2$. 
This is so because for a jump to happen at  $\pm  t^{\chi^{\prime}}/2$, a particle or a hole would have travel a distance of at least 
 $\pm  t^{\chi^{\prime}}/2-t^{\delta}$ during $[t-t^{\chi},t]$ which is very unlikely because $ t^{\chi^{\prime}}/2 \gg t^{\chi}$ and particles/holes  have bounded speed.
 See also Figure \ref{Graph}.

Consequently, $X(t),\tilde{X}(t)$ do not depend on what happens outside $\{j:|j|\leq t^{\chi^{\prime}}/2\}$ during $[t-t^{\chi},t]$. But since $X(t-t^{\chi})=\tilde{X}(t-t^{\chi})$ and 
$\tilde{\eta}_{t-t^{\chi}},\eta_{t-t^{\chi}}$ agree on $\{j:|j|\leq t^{\chi^{\prime}}/2\}$ by \eqref{guilt}, this implies $X(t)=\tilde{X}(t)$.

We now make this argument precise. We can formalize the event that $\eta,\tilde{\eta}$ have no jump involving the sites $\pm  t^{\chi^{\prime}}/2$ as 
\begin{align*}
&\tilde{\mathcal{E}}_{t}=\{\mathrm{\,for\, all \, }\ell \in [t-t^{\chi},t] \mathrm{\,and\,}i\in \{1,2\}, \tilde{\eta}_{\ell}((-1)^{i} t^{\chi^{\prime}}/2)= \tilde{\eta}_{t-t^{\chi}}((-1)^{i} t^{\chi^{\prime}}/2)\}
\\&\mathcal{E}_{t}=\{\mathrm{\,for\, all \, }\ell \in [t-t^{\chi},t] \mathrm{\,and\,}i\in \{1,2\}, \eta_{\ell}((-1)^{i} t^{\chi^{\prime}}/2)= \eta_{t-t^{\chi}}((-1)^{i} t^{\chi^{\prime}}/2)\}.
\end{align*}

We wish to show 
\begin{equation}\label{007}
\lim_{t\to\infty}\Pb(\mathcal{F}_{L,R}^{\delta}\cap \mathcal{E}_{t}^{c})+\Pb(\mathcal{F}_{L,R}^{\delta}\cap \tilde{\mathcal{E}}_{t}^{c})=0.
\end{equation}
Let us show 
\begin{equation}\label{1}
\lim_{t\to\infty}\Pb(\mathcal{F}_{L,R}^{\delta}\cap\{ \mathrm{\,for\, a\, }\ell \in [t-t^{\chi},t],  \eta_{\ell}(- t^{\chi^{\prime}}/2)\neq \eta_{t-t^{\chi}}(- t^{\chi^{\prime}}/2)\})=0.
\end{equation}
To see this, note
\begin{equation*}
\mathcal{F}_{L,R}^{\delta}\subseteq\{x_{L+1}(t-t^{\chi})\leq-t^{\chi^{\prime}}\}\cap\{x_{L}(t-t^{\chi})\geq-t^{\delta}\}\cap \{X(t-t^{\chi})\geq-t^{\delta}\}.
\end{equation*}
So for  the event
\begin{equation*}
\mathcal{F}_{L,R}^{\delta}\cap\{ \mathrm{\,for\, a\, }\ell \in [t-t^{\chi},t],  \eta_{\ell}(- t^{\chi^{\prime}}/2)\neq \eta_{t-t^{\chi}}(- t^{\chi^{\prime}}/2)\}
\end{equation*}
to hold, $x_{L+1}$ or $x_{L}$ or the second class particles would have to make $ t^{\chi^{\prime}}/2-t^{\delta}$ jumps during $[t-t^{\chi},t]$. 
The number of jumps made by  $x_{L+1}, x_{L}$ and the second class particle  can be (crudely) bounded by a rate $2$ Poisson process. 
Since $\chi^{\prime}>\chi, $ we thus  see that \eqref{1} holds. To finish the proof of \eqref{007},  we have to show \eqref{1} with $\eta$ replaced by $\tilde{\eta}$ and/or $- t^{\chi^{\prime}}/2$
replaced by $ t^{\chi^{\prime}}/2$. The required argument for this being identical, we omit this.

Finally, using \eqref{guilt}, we obtain 
\begin{equation*}
\mathcal{F}_{L,R}^{\delta} \cap \tilde{\mathcal{E}}_{t}\cap\mathcal{E}_{t}\subseteq \{ \tilde{\eta}_{t}(j)=\eta_{t}(j), \, \mathrm{ for}  \, |j|\leq t^{\chi^{\prime}}/2\},
\end{equation*}
in particular, $X(t)= \tilde{X}(t)$ if $  \mathcal{F}_{L,R}^{\delta} \cap \tilde{\mathcal{E}}_{t}\cap\mathcal{E}_{t}$ holds. Hence, using \eqref{007}, we may conclude
\begin{align*}
&\lim_{t\to \infty}\Pb(\mathcal{F}_{L,R}^{\delta} \cap\{X(t)\neq \tilde{X}(t)\})
\\&=\lim_{t\to \infty}\Pb(\mathcal{F}_{L,R}^{\delta} \cap \tilde{\mathcal{E}}_{t}\cap\mathcal{E}_{t} \cap\{X(t)\neq \tilde{X}(t)\})=0.
\end{align*}

\end{proof}

Finally, we will squeeze  $\tilde{X}(t)$ in between the leftmost particle and the right most hole of countable state space ASEPs, which we can control using Proposition \ref{DOIT}.

\begin{prop} \label{expprop}There are constants $C_1, C_2 >0 $ such that  we have 
\begin{align*}
\lim_{t\to\infty}\sum_{L,R\in \{0,\ldots,2M\}} \Pb\left(\{|X(t)-R+L|\geq M^{\varepsilon}\}\cap B_{L}\cap D_{R}\right)\leq C_{1} e^{-C_{2} M^{\varepsilon}}.
\end{align*}
\end{prop}
\begin{proof}
We first note that using Propositions \ref{deltaprp} and \ref{Xtild} that
\begin{align*}
&\lim_{t\to\infty}\sum_{L,R\in \{0,\ldots,2M\}} \Pb\left(\{|X(t)-R+L|\geq M^{\varepsilon}\}\cap B_{L}\cap D_{R}\right)
\\&=\lim_{t\to\infty}\sum_{L,R\in \{0,\ldots,2M\}} \Pb\left(\{|\tilde{X}(t)-R+L|\geq M^{\varepsilon}\} \cap \mathcal{F}_{L,R}^{\delta}   \right).
\end{align*}
The main goal of the proof are the bounds \eqref{bounds3}, which   squeeze $\tilde{X}(t)$ in between a particle and a hole of countable state space ASEPs.

Counting yields
\begin{equation}
\mathcal{F}_{L,R}^{\delta} \subseteq\{\tilde{\eta}^{1}_{t-t^{\chi}} \in \Omega_{R-L-1}, \tilde{\eta}^{2}_{t-t^{\chi}}\in \Omega_{R-L}\}.
\end{equation}

Furthermore, by construction the configurations $\tilde{\eta}^{1}_{t-t^{\chi}},\tilde{\eta}^{2}_{t-t^{\chi}} $ have their leftmost particle  to the right of $-t^{\delta}$, and their rightmost hole to the left of $t^{\delta}$.
Consequently, defining
\begin{align*}
&\bar{\eta}^{1}=\mathbf{1}_{\{-t^{\delta},\ldots,L-R\}}+\mathbf{1}_{\{j>t^{\delta}\}}
&\bar{\eta}^{2}=\mathbf{1}_{\{-t^{\delta},\ldots,L-R-1\}}+\mathbf{1}_{\{j>t^{\delta}\}}
\end{align*}
we have  - recalling the partial order \eqref{order} -
\begin{equation}\label{bound}
\mathcal{F}_{L,R}^{\delta} \subseteq\{\bar{\eta}^{1}\preceq \tilde{\eta}^{1}_{t-t^{\chi}}, \bar{\eta}^{2}\preceq \tilde{\eta}^{2}_{t-t^{\chi}}\}.
\end{equation}
Start now at time $t-t^{\chi}$ two ASEPs from $\bar{\eta}^{1}_{t-t^{\chi}}:=\bar{\eta}^{1},\bar{\eta}^{2}_{t-t^{\chi}}:=\bar{\eta}^{2}$.
Let $\bar{H}_{0}^{2}(t)$ resp. $\tilde{H}_{0}^{2}(t)$ be the positions of the rightmost hole of  $\bar{\eta}^{2}_{t}$ resp. $\tilde{\eta}^{2}_{t}$ and 
let $\bar{x}_{0}^{1}(t)$ resp. $\tilde{x}_{0}^{1}(t)$ be the positions of the leftmost particle of  $\bar{\eta}^{1}_{t}$ resp. $\tilde{\eta}^{1}_{t}.$

Now the site $\tilde{X}(t)$ is always occupied by a particle from $\tilde{\eta}^{1}_{t}$ and a hole from $\tilde{\eta}^{2}_{t}$:
 $\tilde{\eta}^{1}_{t}(\tilde{X}(t))=1,\tilde{\eta}^{2}_{t}(\tilde{X}(t))=0.$ Because of this, we get the bounds 
\begin{equation}\label{bound2}
\mathcal{F}_{L,R}^{\delta} \subseteq\{\tilde{x}_{0}^{1}(t)\leq \tilde{X}(t)\leq \tilde{H}_{0}^{2}(t)\}.
\end{equation}

But now we can apply Lemma \ref{lem} to \eqref{bound}, which together with \eqref{bound2} yields

\begin{equation}\label{bounds3}
\mathcal{F}_{L,R}^{\delta} \subseteq\{\bar{x}_{0}^{1}(t)\leq \tilde{X}(t)\leq \bar{H}_{0}^{2}(t)\}.
\end{equation}
To control $\bar{x}_{0}^{1}(t),\bar{H}_{0}^{2}(t),$ we can now use Proposition \ref{DOIT}: We set $a=-t^{\delta},b=L-R$ (resp. $b=L-R-1$), $N=t^{\delta}$ and  $\varepsilon=1$
and note $t^{\chi}>K\mathcal{M}$ for any constant $K$  and $t$ large enough (recall $\mathcal{M}=\max\{N-b,b-a+1\}$). Then, Proposition \ref{DOIT} yields 
\begin{align*}
&\lim_{t\to\infty}\Pb(\bar{H}_{0}^{2}(t)-R+L>M^{\varepsilon})\leq C_{1}e^{-C_{2}M^{\varepsilon}}
\\&\lim_{t\to\infty}\Pb(\bar{x}_{0}^{1}(t)-R+L<-M^{\varepsilon})\leq C_{1}e^{-C_{2}M^{\varepsilon}}.
\end{align*}
Furthermore, note that by construction, $\bar{x}_{0}^{1}(t),\bar{H}_{0}^{2}(t)$ are independent of $\mathcal{F}_{L,R}^{\delta}$.
We conclude that
\begin{align*}
&\lim_{t\to\infty}\sum_{L,R\in \{0,\ldots,2M\}} \Pb\left(\{\tilde{X}(t)-R+L\geq M^{\varepsilon}\} \cap \mathcal{F}_{L,R}^{\delta}   \right)
\\&\leq\lim_{t\to\infty}\sum_{L,R\in \{0,\ldots,2M\}} \Pb\left(\{\bar{H}_{0}^{2}(t)-R+L\geq M^{\varepsilon}\} \cap \mathcal{F}_{L,R}^{\delta}   \right)
\\&=\lim_{t\to\infty}\sum_{L,R\in \{0,\ldots,2M\}} \Pb\left(\{\bar{H}_{0}^{2}(t)-R+L\geq M^{\varepsilon}\}\right)\Pb(  \mathcal{F}_{L,R}^{\delta}   )
\\&\leq\lim_{t\to\infty}\sum_{L,R\in \{0,\ldots,2M\}}  C_{1}e^{-C_{2}M^{\varepsilon}}\Pb(  \{\mathcal{P}=L\}\cap\{\mathcal{H}=R\}   )
\\&\leq  C_{1}e^{-C_{2}M^{\varepsilon}}.
\end{align*}
Likewise, we have that
\begin{align*}
&\lim_{t\to\infty}\sum_{L,R\in \{0,\ldots,2M\}} \Pb\left(\{\tilde{X}(t)-R+L\leq- M^{\varepsilon}\} \cap \mathcal{F}_{L,R}^{\delta}   \right)
\\&\leq\lim_{t\to\infty}\sum_{L,R\in \{0,\ldots,2M\}} \Pb\left(\{\bar{x}_{0}^{1}(t)-R+L\leq- M^{\varepsilon}\} \cap \mathcal{F}_{L,R}^{\delta}   \right)
\\&\leq  C_{1}e^{-C_{2}M^{\varepsilon}}.
\end{align*}
This finishes the proof.
\end{proof}

\begin{proof}[Proof of Theorem \ref{couplthm}]
Using Propositions \ref{redprop} and \ref{expprop} we obtain
\begin{align*}
&\lim_{M\to \infty}\lim_{t\to\infty}\Pb\left( |X(t)-\mathcal{H}+\mathcal{P}|>M^{\varepsilon}\right)\\&=\lim_{M\to \infty}\lim_{t\to\infty}\sum_{L,R\in \{0,\ldots,2M\}}\Pb\left(\{|X(t)-R+L|\geq M^{\varepsilon}\}\cap B_{L}\cap D_{R}\right)=0.
\end{align*}
\end{proof}
\section{Acknowledgements}
This work is supported by the Deutsche Forschungsgemeinschaft
(German Research Foundation) by the CRC 1060 (Projektnummer
211504053) and Germany's Excellence Strategy - GZ 2047/1, Projekt ID
390685813. Part of this work was done while the author was affiliated with IST Austria, where his research was supported by ERC Advanced Grant No. 338804 and ERC Starting Grant No. 716117.
\bibliography{Biblio}{}
\bibliographystyle{plain}
\end{document}